\documentclass[a4paper,english,thm-restate,colorlinks=true,linkcolor=blue,citecolor=blue,urlcolor=blue]{lipics-v2021}
\nolinenumbers
\usepackage{amsfonts}
\usepackage{amssymb}
\usepackage{amsthm}
\usepackage{graphicx} 
\usepackage{url}
\usepackage{tcolorbox} 
\usepackage{hyperref} 
\usepackage{todonotes}
\hideLIPIcs

\newtheorem{problem}[theorem]{Problem}

\newcommand{\Vg}{V_\mathrm{g}}
\newcommand{\Vo}{V_\mathrm{o}}
\newcommand{\Vp}{V_\mathrm{p}}

\ccsdesc[500]{Theory of computation~Computational geometry}

\title{Two Results on Outer-String Graphs}
\keywords{String graph, Outer-string graph, Constrained outer-string graph, Outer-1-string graph, Computational complexity, NP-hardness}

\author{Todor Anti\'{c}}{Department of Applied Mathematics, Faculty of Mathematics and Physics, Charles University, Czech Republic}{todor@kam.mff.cuni.cz}{https://orcid.org/0009-0008-6521-7987}{Research supported by Grant Agency of Charles University (GAUK) grant No. 302226 and  Czech Science Foundation grant No. GA\v CR 23-04949X.}
\author{V\'{\i}t Jel\'{\i}nek}{Computer Science Institute, Faculty of Mathematics and Physics, Charles University, Prague, Czech Republic}{jelinek@iuuk.mff.cuni.cz}{https://orcid.org/0000-0003-4831-4079}{Research supported by Czech Science Foundation grant No. GA\v CR 23-04949X.}
\author{Jan Kratochv\'{\i}l}{Department of Applied Mathematics, Faculty of Mathematics and Physics, Charles University, Czech Republic}{honza@kam.mff.cuni.cz}{https://orcid.org/0000-0002-2620-6133}{Research supported by Czech Science Foundation grant No. GA\v CR 23-04949X.}
\author{Peter Stumpf}{Department of Applied Mathematics, Faculty of Mathematics and Physics, Charles University, Czech Republic\\
Department of Theoretical Computer Science, Faculty of Information Technology, Czech Technical University in Prague, Czech Republic}{stumpf@kam.mff.cuni.cz}{https://orcid.org/0000-0003-0531-9769}{Research supported by Czech Science Foundation grant No. GA\v CR 23-04949X.}

\authorrunning{T.~Anti\'{c}, V.~Jel\'{\i}nek, J.~Kratochv\'{\i}l, P.~Stumpf} 

\Copyright{Todor~Anti\'{c}, V\'{\i}t~Jel\'{\i}nek,Jan~Kratochv\'{\i}l and Peter Stumpf} 

\acknowledgements{{This research was initiated at the DiGeo Workshop 2025.}}

\EventEditors{Maarten L\"{o}ffler and Silvia Miksch}
\EventNoEds{2}
\EventLongTitle{34th International Symposium on Graph Drawing and Network Visualization (GD 2026)}
\EventShortTitle{GD 2026}
\EventAcronym{GD}
\EventYear{2026}
\EventDate{August 19--21, 2026}
\EventLocation{Ontario, Canada}
\EventLogo{}
\SeriesVolume{396}
\ArticleNo{26}

\begin{document}

\maketitle

\begin{abstract}
An \emph{outer-string representation} of a graph~$G$ is an intersection representation of~$G$ where vertices are represented by curves (strings) inside the unit disk, and each curve has exactly one endpoint on the boundary of the unit disk (the anchor of the curve). Additionally, if any two curves are allowed to cross at most once, we call this an \emph{outer-$1$-string representation} of~$G$. If we impose a cyclic ordering on the vertices of~$G$ and require the cyclic order of the anchors to respect this cyclic order, such a representation is called a \emph{constrained outer-string representation}. 
In this paper, we present two results about graphs admitting outer-string representations. 

Firstly, we show that for a bipartite graph~$G$ (and, more generally, for any~$\{C_3,C_5\}$-free graph~$G$) with a given cyclic order of vertices, we can decide in polynomial time whether~$G$ admits a constrained outer-string representation. Our algorithm follows from a characterization by a single forbidden configuration, similar to that of Biedl et al. [GD 2024] for chordal graphs. Secondly, we answer an open question from the same authors and show that determining whether a given graph admits an outer-1-string representation is NP-hard. More generally, we show that it is NP-hard to determine if a given graph~$G$ admits an outer-$k$-string representation for any fixed~$k\ge1$.

\subparagraph{Generative AI declaration} Generative AI was not used in the preparation of this article.
\end{abstract}
\newpage
\section{Introduction}\label{sec:intro}

Geometric intersection graphs are intensively studied both for their applied motivation and for interesting structural and algorithmic properties. If~$\cal M$ is a universe of (typically arc-connected) sets in the Euclidean plane defined by their geometrical properties, an~$\cal M$ \emph{representation} of a graph is an assignment of sets from~$\cal M$ to the vertices of the graph such that any two vertices are adjacent if and only if the sets assigned to them are non-disjoint (i.e., they intersect). A graph is an~$\cal M$-\emph{intersection graph} (or briefly an~$\cal M$ \emph{graph}) if it allows an~$\cal M$ representation. Classical and well studied examples are interval graphs (intersection graphs of intervals on a line), circular arc graphs (intersection graphs of arcs of a circle), disk graphs (intersection graphs of disks in the plane), and many others. Cf. the monographs~\cite{golumbic:perfect,B-L-Spinrad,McKee}. \emph{String graphs} (intersection graphs of curves in the plane) form the most general of intersection-defined classes of graphs (when arc-connected sets in the plane are considered)~\cite{ehrlichET76,sinden66}. Their recognition is NP-complete; cf.~\cite{Kratochvil91a} for NP-hardness and~\cite{jcss-SchaeferSS03} for NP-membership.      

In this paper, we are concerned with \emph{outer-string graphs}, a subclass of string graphs. An \emph{outer-string representation} is a string representation such that all the curves of the representation reach the outerface. Equivalently, all curves representing the vertices of the intersection graph are required to lie within a predefined region (a disk or a half-plane) and be attached to the boundary of the region (the bounding circle or the boundary line) by one end-point, which is referred to as the \emph{anchor} of the curve. The outer-string graphs were named as such in~\cite{kratochvil:82}, but  were already implicitly defined  in~\cite{sinden66}. In that seminal paper, Sinden provided a motivation stemming from thin film RC-circuits, which were predecessors of VLSI designs. He studied in detail the \emph{constrained} version of the outer-string recognition problem, namely, the variant in which the anchors of the curves must respect a prescribed cyclic order on the boundary of the region that hosts the representation. He identified an infinite sequence of critical graphs that do not admit a constrained outer-string representation, specifically the complements of cycles of lengths greater than 3 with the natural cyclic order of anchors. 
In~\cite{biedlCKR:gd24,biedlCKR:DAM26}, the authors show that if $G$ is a chordal graph with a prescribed cyclic order of its vertices, then $G$ admits a constrained outer-string representation if and only if it does not contain the complement of  such a cycle as an induced subgraph. Moreover, since the complement of a cycle of length greater than 4 contains~$C_4$ or~$C_5$ as an induced subgraph, a chordal graph (with a prescribed cyclic order of the vertices) admits a constrained outer-string representation if and only if it does not contain two alternating independent edges, which is the complement of~$C_4$ with the natural cyclic order of its vertices.

It is important to note that in an outer-string representation (and more generally in a string representation), the number of crossing points of a pair of strings is not controlled. Kratochv\'{\i}l and Matou\v{s}ek~\cite{KratochvilM91exp} showed that there are string graphs that require an exponential number of crossing points in any string representation. Motivated by this example, Biedl, Biniaz, and Derka~\cite{BiedlBD18} showed that there are outer-string graphs that require an exponential number of crossing points in any outer-string representation. Their example can be easily modified to be bipartite and immediately implies that an exponential number of crossings is needed in constrained outer-string representations. This has led to the definition of~$k$-string (outer-$k$-string) graphs that admit a string (an outer-string, respectively) representation such that any two curves share at most~$k$ crossing points (a touching point is counted twice). Outer-1-string graphs, in particular, were introduced and studied in~\cite{BiedlD17}. Although 1-string graphs are NP-complete to recognize~\cite{Kratochvil91a}, the computational complexity of recognizing outer-1-string graphs is asked as an open problem in~\cite{BiedlBD18,biedlCKR:gd24,biedlCKR:DAM26}. To complete the picture, let us mention that it is shown in~\cite{Kratochvil91a} that the NP-hardness of recognizing outer-string graphs follows from the proof of the NP-hardness of recognizing the so called cylinder graphs presented in~\cite{MiddendorfP93cylinder}. However, that reduction does not control the number of crossing points.

\medskip\noindent
{\bf Our results} We first consider the constrained outer-string graphs. In Theorem~\ref{thm:C3C5-free}, we show that a~$\{C_3,C_5\}$-free graph admits a constrained outer-string representation if and only if it does not contain an alternating pair of independent edges. In particular, the existence of a constrained outer-string representation can be decided in polynomial time for such graphs, even though the representation may require exponentially many crossings. This is a counterpart to the above mentioned characterization for chordal graphs~\cite{biedlCKR:gd24,biedlCKR:DAM26}. It is somewhat interesting that for both chordal and~$\{C_3,C_5\}$-free graphs, the critical constrained non-outer-string instance is unique and the same. An immediate corollary of our result is that the same holds for bipartite graphs.

Our second main result concerns outer-string graphs with a bounded number of crossing points per pair of strings. In Theorem~\ref{thm:nphardness}, we show that recognizing outer-1-string graphs is NP-complete, thus answering a question explicitly asked in~\cite{biedlCKR:gd24,biedlCKR:DAM26}.

\section{Preliminaries and Notation}\label{sec:prelim}

We consider simple undirected graphs (i.e., no loops or multiple edges), with edges being two-element subsets of the set of vertices. For a graph~$G$, its vertex set (edge set) will be denoted by~$V(G)$ ($E(G)$, respectively). For the sake of brevity, we use the notation~$uv$ (instead of~$\{u,v\}$) for the edge joining vertices~$u$ and~$v$. 
For an outer-string graph~$G$, we usually denote an outer-string representation of~$G$ by~$G^*$. We will assume that~$G^*$ is a representation of~$G$ by non-self-crossing curves grounded on the boundary of the unit disk (which we denote by~$D$) and with all curves contained in the interior of~$D$. Furthermore, for a  curve~$\gamma$ in~$G^*$, we assume that~$\gamma$ intersects the boundary of~$D$ in exactly one point, which we call the \emph{anchor} of~$\gamma$ and denote by~$\gamma^\bot$. We call the other endpoint of~$\gamma$ the \emph{tip} of~$\gamma$, and we denote it by~$\gamma^\top$. For a vertex~$v\in V(G)$, we denote by~$v^*$ the curve representing~$v$ in~$G^*$ and denote the anchor and tip of~$v^*$ by~$v^\bot$ and~$v^\top$, respectively. 

We call a graph~$G$ together with a cyclic order~$\pi$ of its vertices an \emph{ordered graph}.
An ordered graph~$(G,\pi)$ is a \emph{solvable constrained outer-string instance} if~$G$ has an outer-string representation such that the cyclic ordering of the anchors coincides with~$\pi$.
We say that an ordered graph~$F$ is an \emph{outer-string obstruction} if it does not admit a constrained outer-string representation, while every induced ordered subgraph of~$F$ admits such a representation. Clearly, if~$F$ is an outer-string obstruction and an ordered  graph~$G$ contains~$F$ as an induced subgraph, then~$G$ does not admit a 
constrained outer-string representation.

\begin{figure}[h]
    \centering
    \includegraphics[scale = 0.8]{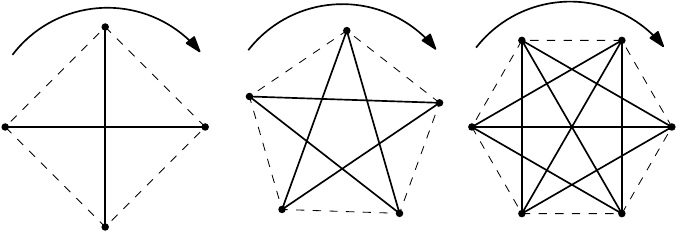}
    \caption{The co-holes~$-C_n$ for~$n\in\{4,5,6\}$ with the natural order of vertices. }
    \label{fig:obstruction}
\end{figure}

The complement~$-C_n$ of the cycle~$C_n$ of length~$n$ ordered cyclically in the natural order is called a \emph{co-hole} (of length~$n$). See Figure \ref{fig:obstruction}.

\begin{lemma}[Sinden \cite{sinden66}] \label{lem:co-hole}
    For every $n\ge 4$, the co-hole of length $n$ is an outer-string obstruction.
\end{lemma}

 A co-hole of length~$4$ is called an~\emph{alternating-edge obstruction} in~\cite{biedlCKR:gd24,biedlCKR:DAM26}.  

Next, we prove a simple technical lemma that we will frequently use in our proofs to show that certain configurations cannot appear in  solvable constrained outer-string instances.
Two vertex-disjoint subgraphs~$G_1, G_2$ of a graph~$G$ are called \emph{independent} if~$G$ has no edge incident with vertices from both~$G_1$ and~$G_2$.

\begin{lemma}\label{lem:crossingpaths}
    Let~$G$ be a graph, and let~$P=p_1,p_2,\dots,p_k$ and~$Q=q_1,q_2,\dots,q_s$ be two independent vertex-disjoint paths in~$G$. If~$G$ is ordered so that the end-vertices of the paths alternate, then~$G$ is not a solvable constrained outer-string instance. 
\end{lemma}

\begin{proof}
Lemma~2 of~\cite{biedlCKR:DAM26} says that in a solvable instance, the vertices of no two independent connected subgraphs may alternate. Our situation is a special case of two independent paths with alternating end-vertices.
\end{proof}

\section{The \texorpdfstring{$\{C_3,C_5\}$}{\{C\_3,C\_5\}}-Free Case}\label{sec:C_3-C_5-free}

A graph is~$\{C_3, C_5\}$-free if it does not contain an induced subgraph isomorphic to a cycle of length 3 or 5.
This is the main result of this section:
\begin{theorem}\label{thm:C3C5-free}
If~$G$ is an ordered~$\{C_3,C_5\}$-free graph\footnote{Note that this means that the underlying unordered graph is~$\{C_3,C_5\}$-free.}, then a
constrained outer-string instance for~$G$ is solvable if and only if it does not contain the alternating-edge obstruction.
\end{theorem}

The rest of this section is devoted to the proof of Theorem \ref{thm:C3C5-free}. 

Clearly, an instance that contains the alternating-edge obstruction is not solvable. Assume, therefore, that we are given an ordered~$\{C_3, C_5\}$-free graph~$G$ and a constrained outer-string instance with an underlying graph~$G$ and no alternating-edge obstruction.

We proceed by induction on~$n=|V|$,~$n=1$ being 
trivial. 

Suppose~$n>1$. We will construct an outer-string representation inside the unit disk~$D$, with anchors on the boundary of~$D$. Suppose that the anchors are given and there is no alternating edge obstruction. Choose a vertex~$t\in V$ (``turquoise''). Call the neighbors of~$t$ \emph{green vertices}, the vertices in~$G-t$ with at least one green neighbor \emph{orange}, and the remaining vertices of~$G$~\emph{purple}. Let~$\Vg$,~$\Vo$, and~$\Vp$ denote the sets of green, orange, and purple vertices, respectively. See Figure \ref{fig:verticesofG}.

\begin{figure}[ht]
    \centering
    \includegraphics[page=1]{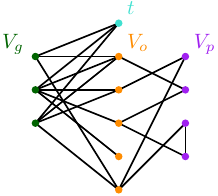}
    \caption{Coloring of the vertices of~$G$ in the proof of Theorem \ref{thm:C3C5-free}.}
    \label{fig:verticesofG}
\end{figure}

Fix a representation~$G^*_-$ of~$G-t$ in~$D$, which exists by induction. 

\begin{lemma}\label{lem-ind}
Both~$\Vg$ and~$\Vo$ are independent sets.
\end{lemma}
\begin{proof}
If there were an edge connecting two green vertices~$g_1$ and~$g_2$, then~$t$,~$g_1$, and~$g_2$ would form a~$C_3$. It follows that the green vertices form an independent set.

Suppose now that there is an edge connecting two orange vertices~$o_1$ and~$o_2$. Let~$g_1$ be a green neighbor of~$o_1$, and let~$g_2$ be a green neighbor of~$o_2$. If~$o_2$ were adjacent to~$g_1$ (which includes the case~$g_1=g_2$), then~$o_1$,~$o_2$, and~$g_1$ would form a~$C_3$. Thus,~$o_2$ is not adjacent to~$g_1$, and by the same argument,~$o_1$ is not adjacent to~$g_2$. However, this shows that~$o_1$,~$o_2$,~$g_1$,~$g_2$, and~$t$ form an induced~$C_5$.
\end{proof}

Let us fix some terminology. If two vertices~$x,y\in\Vo\cup\Vp$ are adjacent,
we say that the two curves~$x^*$ and~$y^*$ form a \emph{pocket}. The pocket partitions the disk~$D$ into several (at least two) faces. One of the faces contains the turquoise anchor~$t^\bot$, and we say that this face is \emph{outside} the pocket, while all the remaining faces are \emph{inside}.

The \emph{OP-graph} (or \emph{orange-purple graph}) is the plane multigraph whose vertices are the anchors and tips of the orange and purple curves, and the crossings between these curves, while its edges are the maximal subcurves of the orange and purple curves that connect two vertices without passing through another. We will refer to the vertices, edges, and faces of the OP-graph as \emph{OP-vertices}, \emph{OP-edges} , and \emph{OP-faces} respectively.

A \emph{barrier} is a path in the OP-graph whose both end-vertices are anchors. Note that a barrier splits~$D$ into two internally disjoint parts. The part that contains~$t^\bot$ is said to be \emph{outside} the barrier, while the other part is \emph{inside}. We say that a barrier is \emph{separating} if there is at least one green anchor inside it. 

We define a \emph{chunk} of a barrier~$\beta$ as a maximal subpath of~$\beta$ whose OP-edges all belong to the same curve~$v^*$ in the representation of~$G$. See Figure \ref{fig:OPgraph}.

\begin{figure}[ht]
    \centering
    \includegraphics[page=2]{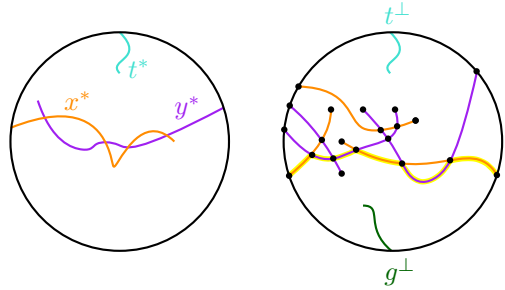}
    \caption{On the left we show a pocket formed by two curves~$x^*,y^*$ with five faces, four of which are inside the pocket. On the right, we show an OP-graph and a separating barrier highlighted in green.}
    \label{fig:OPgraph}
\end{figure}

\begin{lemma}\label{lem-redbarrier}
Every separating barrier contains at least one orange OP-edge. 
\end{lemma}

\begin{proof}
Let~$\beta$ be a separating barrier that contains no orange OP-edge. Let~$A_1$ and~$A_2$ be the two anchors of~$\beta$, and let~$g^\bot$ be a green anchor inside~$\beta$, which exists since~$\beta$ is separating. The two points~$t^\bot$ and~$g^\bot$ partition the boundary of~$D$ into two arcs, one of which contains~$A_1$ (call this arc~$\alpha_1$), and the other contains~$A_2$ (call it~$\alpha_2$). Let~$\gamma$ be a chunk of~$\beta$. We say that~$\gamma$ is of \emph{type 1} if its anchor is in~$\alpha_1$; otherwise, it is of \emph{type 2}. Since~$\beta$ must contain chunks of both types, we can choose two consecutive chunks~$\gamma_1$ and~$\gamma_2$ with~$\gamma_1$ of type 1 and~$\gamma_2$ of type~2. Let~$p_i$ denote the vertex whose curve~$p_i^*$ contains~$\gamma_i$. Since no purple vertex may be adjacent to a green or turquoise vertex, we conclude that~$t$,~$g$,~$p_1$, and~$p_2$ form an alternating edge obstruction. See Figure \ref{fig:barrierbicolored}.\qedhere

\begin{figure}
    \centering
    \includegraphics[page=3]{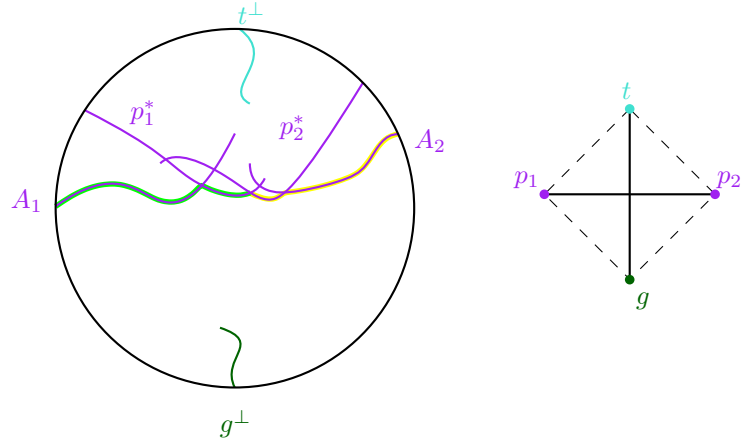}
    \caption{Proof of Lemma \ref{lem-redbarrier}, chunks of type 1 are highlighted in light green and chunks of type 2 are highlighted in pink. Vertices~$p_1,t,p_2,g$ form an alternating-edge obstruction.}
    \label{fig:barrierbicolored}
\end{figure}
\end{proof}

\begin{lemma}\label{lem-barrier}
Every separating barrier~$\beta$ contains an orange OP-edge~$e$ with the following property: if~$o^*$ is the orange curve that contains~$e$, and if~$f$ is the OP-face incident to~$e$ that is inside~$\beta$, then there is a purple curve~$p^*$ forming a pocket with~$o^*$ such that~$f$ is inside the pocket. (Note that~$p^*$ is not necessarily participating in~$\beta$.)
\end{lemma}
\begin{proof}
As in the previous proof, let~$A_1$ and~$A_2$ be the two anchors of~$\beta$, and let~$g^\bot$ be a green anchor inside~$\beta$. The two points~$t^\bot$ and~$g^\bot$ partition the boundary of~$D$ into two arcs, one of which contains~$A_1$ (call this arc~$\alpha_1$), and the other contains~$A_2$ (call it~$\alpha_2$).

We prove the lemma by induction on the number of chunks in~$\beta$. Since~$\beta$ contains two anchors, it must contain at least two chunks, and at least one of these chunks is orange by Lemma~\ref{lem-redbarrier}. If~$\beta$ contains exactly two chunks, then, since $V_o$ forms an independent set, it is contained in the union of an orange curve~$o^*$ and a purple curve~$p^*$. The two curves form a pocket, and any face inside~$\beta$ is also inside this pocket. Therefore the lemma holds.

Assume now that~$\beta$ has at least three chunks. The two chunks of~$\beta$ containing~$A_1$ and~$A_2$ are said to be the \emph{outer chunks}, and the remaining chunks are the \emph{inner chunks}.

We will now distinguish two cases: either~$\beta$ contains an inner purple chunk, or it does not.  Suppose first that there is an inner purple chunk~$\gamma$, let~$C_1$ and~$C_2$ be the endpoints of~$\gamma$, numbered so that the subpath of~$\beta$ between~$A_1$ and~$C_1$ does not contain~$C_2$. In Figure \ref{fig:complicatedcase1}, we illustrate the steps for finding the OP-edge in this case, which we describe next. 

\begin{figure}[ht]
    \centering
    \includegraphics[page=4]{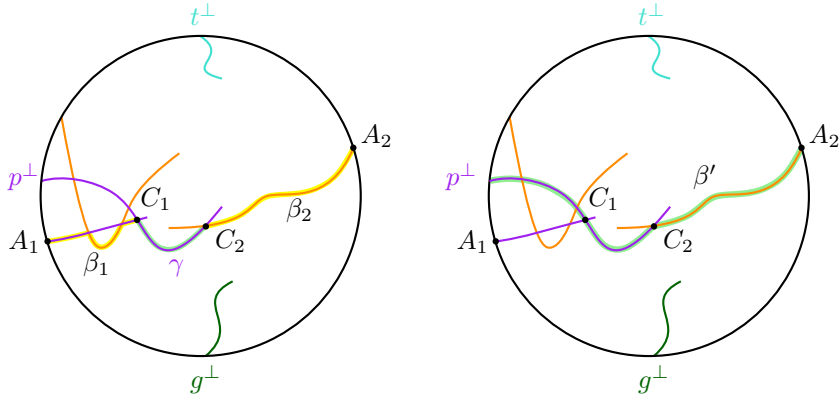}
    \caption{Transforming a separating barrier~$\beta = \beta_1 \cup\gamma\cup \beta_2$ on the left into a different separating barrier~$\beta'$  with a smaller number of chunks (highlighted in green on the right).}
    \label{fig:complicatedcase1}
\end{figure}

Let~$\beta_1$ be the subpath of~$\beta$ between~$A_1$ and~$C_1$, and let~$\beta_2$ be the subpath of~$\beta$ between~$A_2$ and~$C_2$. In particular,~$\beta$ is the union of the three curves~$\beta_1$,~$\gamma$ and~$\beta_2$. 

Let~$p^*$ be the purple curve containing~$\gamma$ as a subset, and let~$p^\bot$ be its anchor. Without loss of generality, assume that~$p^\bot$ is in~$\alpha_1$ (possibly~$p^\bot=A_1$). 

We construct a new separating barrier~$\beta'$ as follows: we start in~$p^\bot$ and follow~$p^*$ until we first reach a point~$X$ belonging to~$\beta_2$ (since~$\beta_2\cap p^*$ contains the point~$C_2$, such a point~$X$ is well defined). From~$X$, we then follow~$\beta_2$ towards~$A_2$, completing the barrier~$\beta'$. The barrier~$\beta'$ is separating, since~$g^\bot$ is inside it.

Moreover, we observe that~$\beta'$ has the following properties:
\begin{itemize}
    \item it has strictly fewer chunks than~$\beta$,
    \item any orange OP-edge of~$\beta'$ is also an OP-edge of~$\beta$, and 
    \item for any orange OP-edge~$e$ of~$\beta'$, the OP-face incident to~$e$ which is inside~$\beta'$ is also inside~$\beta$.
\end{itemize}
We may now apply induction to~$\beta'$, which yields an orange OP-edge~$e$ of~$\beta'$ contained in an orange curve~$o^*$, such that there is a purple curve~$p^*$ forming a pocket with~$o^*$, and the OP-face~$f$ incident to~$e$ which is inside the barrier~$\beta'$ is also inside the pocket of~$o^*\cup p^*$. As observed above, the OP-edge~$e$ also belongs to~$\beta$, and the face~$f$ is also inside the barrier~$\beta$. Thus,~$e$,~$o^*$ and~$p^*$ witness that the lemma holds for~$\beta$ as well. This completes the case when~$\beta$ has an inner purple chunk.

Suppose now that~$\beta$ has no inner purple chunk. Since we assume that~$\beta$ has at least three chunks, and since no two orange chunks can appear consecutively by Lemma~\ref{lem-ind}, it follows that~$\beta$ has exactly three chunks, with the two outer chunks being purple and the sole inner chunk being orange. Call the three chunks~$\beta_1$,~$\gamma$ and~$\beta_2$ in the order from~$A_1$ to~$A_2$, and let~$p_1^*$,~$o^*$ and~$p^*_2$ be the three curves containing~$\beta_1$,~$\gamma$ and~$\beta_2$, respectively. 

Let~$e$ be any orange OP-edge contained in~$\gamma$, and let~$f$ be the OP-face incident to~$e$ which is inside~$\beta$. The curve~$o^*$ forms a pocket with~$p_1^*$ as well as with~$p_2^*$. If~$f$ is inside at least one of these pockets, we are done, so suppose for contradiction that~$f$ is outside both of them.

Fix a point~$F$ in the interior of the OP-face~$f$. Since~$f$ is outside the pocket of~$o^*\cup p_1^*$, we can draw a curve~$\delta_1$ connecting~$F$ to~$t^\bot$ without crossing~$o^*\cup p_1^*$. Let us choose such~$\delta_1$ in such a way that it avoids the boundary of~$D$ except the point~$t^\bot$, and crosses~$\beta_2$ a finite number of times. By construction,~$\delta_1$ avoids~$\beta_1\cup \gamma$, and since~$F$ is inside~$\beta$ while~$t^\bot$ is outside,~$\delta_1$ crosses~$\beta_2$ an odd number of times. Similarly, draw a curve~$\delta_2$ connecting~$F$ to~$t^\bot$ while avoiding~$o^*\cup p_2^*$, avoiding the boundary of~$D$ except in~$t^\bot$, and crossing~$\beta_1$ an odd number of times.
We may further choose~$\delta_2$ to only cross~$\delta_1$ a finite number of times. 

Now consider the (possibly self-intersecting) closed curve~$\delta:=\delta_1\cup\delta_2$, and let~$G_\delta$ be the plane multigraph whose vertices are the intersection points of~$\delta_1$ and~$\delta_2$ (including~$F$ and~$t^\bot$) and whose edges are the subcurves of~$\delta$ connecting two vertices without passing through another vertex.~$G_\delta$ only has vertices of even degree ($F$ and~$t^\bot$ have degree 2, all the other vertices have degree 4). Therefore its dual~$G^*_\delta$ is bipartite. Note that the entire boundary of~$D$ except the point~$t^\bot$ belongs to the outer face of~$G_\delta$. Furthermore, since~$o^*$ never intersects~$\delta$, the chunk~$\gamma$ also belongs to the outer face of~$G_\delta$. Thus, the chunk~$\beta_1$ has both endpoints in the outer face of~$G_\delta$, while crossing~$\delta$ an odd number of times. This means that the dual~$G^*_\delta$ has a closed walk of odd length, which is impossible, since~$G^*_\delta$ is bipartite. This contradiction completes the proof.
\end{proof}

The \emph{DOP-graph} (or \emph{partial dual-OP-graph}) is the directed plane multigraph whose vertices are the OP-faces, and the directed edges are defined as follows: for every orange OP-edge~$e$, let~$f$ and~$f'$ be the two OP-faces separated by~$e$. If the orange curve~$o*$ containing~$e$ is part of a pocket which contains~$f$ on its inside, then we add the directed edge~$(f,f')$ to the DOP-graph. (Note that we may also add the opposite edge~$(f',f)$, in fact the same pocket of~$e$ may contain both~$f$ and~$f'$ on the inside.) We will draw the directed edge~$(f,f')$ in such a way that it crosses its ``primal'' edge~$e$ and no other OP-edge, see Figure \ref{fig:DOP}.

\begin{figure}
    \centering
    \includegraphics[page=5]{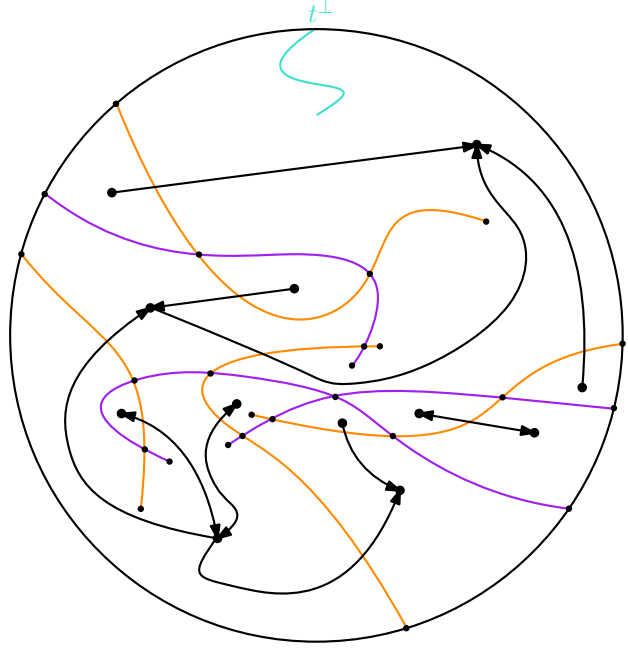}
    \caption{DOP-graph of~$G$, for readability purposes, when both directed edges are present between the vertices, we draw them as a single edge oriented both ways.}
    \label{fig:DOP}
\end{figure}

Note that the DOP-graph does not depend on the green vertices.

Notice that Lemma~\ref{lem-barrier} can be reformulated as follows.

\begin{lemma}\label{lem-barrier2}
For any separating barrier~$\beta$, the DOP-graph contains a directed edge crossing~$\beta$ and oriented from the inside towards the outside of~$\beta$.
\end{lemma}

We now, at last, take the green vertices into account. Let~$g$ be a green vertex. A \emph{trajectory} of~$g$ is a directed curve~$T(g)$ with the following properties:
\begin{itemize}
    \item~$T(g)$ starts in the anchor~$g^\bot$ and ends in a point inside the OP-face containing~$t^\bot$.
    \item~$T(g)$ does not intersect any purple curve.
    \item~$T(g)$ can only cross an orange edge~$e$ in a way which is consistent with the DOP-graph; more precisely, if~$T(g)$ exits from an OP-face~$f_1$ and enters an OP-face~$f_2$ by crossing an orange OP-edge~$e$, then the DOP-graph must contain the directed edge~$(f_1,f_2)$.
\end{itemize}

\begin{lemma}
Let~$g$ be a green vertex. Then the DOP-graph contains a directed path from the OP-face that contains the anchor~$g^\bot$ to the OP-face containing the anchor~$t^\bot$. In particular,~$g$ has a trajectory.
\end{lemma}
\begin{proof}
Suppose the lemma fails. Let~$R$ be a region of~$D$ obtained as the union of all the closed OP-faces that can be reached by an oriented path in the DOP-graph from the OP-face containing~$g^\bot$. We call the faces in~$R$ \emph{reachable}. Let~$U$ be similarly the union of all the unreachable closed OP-faces. Note that~$U$ may fail to be topologically connected. Let~$U_t$ be the topological connectivity component of~$U_t$ that contains~$t^\bot$ on its boundary; in other words,~$U_t$ is a minimal subset of~$U$ with the property that the closed OP-face containing~$t^\bot$ is in~$U_t$, and for any face contained in~$U_t$, all its adjacent unreachable OP-faces are in~$U_t$. 

The shared boundary of~$R$ and~$U_t$ contains a barrier~$\beta$ with~$R$ being inside~$\beta$. Since~$g^\bot$ is on the boundary of~$R$,~$\beta$ is a separating barrier. By Lemma~\ref{lem-barrier2}, the DOP-graph contains a directed edge crossing~$\beta$ from the inside to the outside, contradicting the definition of~$R$.
\end{proof}

We now construct an outer-string representation of~$G$. To do this, we will modify the representation of the green vertices in such a way that each green vertex is represented by a curve that intersects the OP-face containing~$t^\bot$.

First, we fix a trajectory~$T(g)$ for every green vertex~$g$. Then, processing the green vertices sequentially one by one, we perform the following steps for each green vertex~$g$:
\begin{itemize}
    \item  Extend~$g^*$ by retracing it from its tip back to its anchor~$g^\bot$.
    \item Add~$T(g)$ to~$g^*$ by ``digging a tunnel'' along~$T(g)$ starting from~$g^\bot$ and proceeding as follows. Every time the tunnel reaches a point where~$T(g)$ crosses a curve representing a green vertex~$g_2$ (possibly equal to~$g$), the curve~$g^*_2$ is ``pushed in front'' of the tip of the tunnel, to avoid creating intersections between green curves, see Figure \ref{fig:tunneling}. 
    \item The process for~$g$ stops, when its tunnel, and all the green curves being pushed in front of it, reaches into the interior of the OP-face containing~$t^\bot$.
    \item Once all the green vertices have been processed, we represent~$t$ by an arbitrary curve~$t^*$ starting in~$t^\bot$ and crossing all the green curves without crossing any purple or orange curve.
\end{itemize}

\begin{figure}
    \centering
    \includegraphics[page=6,scale=0.9]{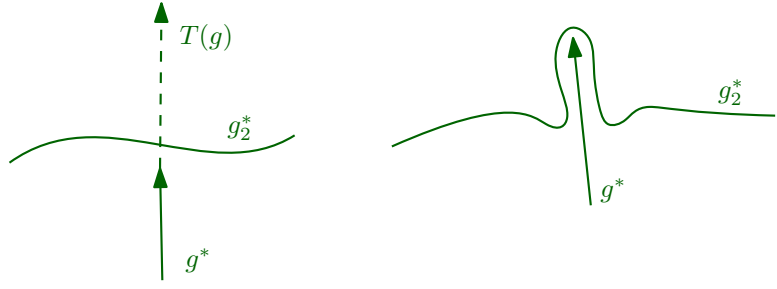}
    \caption{Extending~$g^*$ along its trajectory~$T(g)$,~$g_2^*$ is being ``pushed in front'' of~$g$.}
    \label{fig:tunneling}
\end{figure}

It remains to argue that the resulting arrangement of curves is a valid representation of~$G$. Clearly all the edges of~$G$ are represented. We thus only need to show that whenever the above construction creates an intersection between a green curve and an orange curve, then the two corresponding vertices are adjacent in~$G$. 

Suppose this is not the case, and let us focus on the moment in the construction right before the first illegal crossing is generated. Suppose the illegal crossing involves a green vertex~$\bar g$ and an orange vertex~$o$, and happens in the process of digging a tunnel for a green vertex~$g$ (possibly equal to~$\bar g$), which is about to cross an orange OP-edge~$e\subseteq o^*$ from an OP-face~$f_1$ to an OP-face~$f_2$. By construction, the DOP-graph must therefore contain a directed edge~$(f_1,f_2)$, and by the DOP-graph's definition, there must be a purple curve~$p^*$ forming a pocket with~$o^*$ with~$f_1$ inside the pocket.

We now distinguish two possibilities, depending on the position of the anchor~$\bar g^\bot$ of~$\bar g$. If the anchor is inside the pocket of~$o^*$ and~$p^*$, then the four vertices~$\bar g, o, p, t$ form an alternating-edge obstruction, a contradiction.

Therefore,~$\bar g^\bot$ is outside the pocket of~$o^*$ and~$p^*$. However, immediately before creating the first illegal green-orange crossing, the curve~$g^*$ is inside the face~$f_1$, which is inside the pocket. Therefore,~$\bar g^*$ must already have crossed~$o^*$ before the first illegal crossing was created, a contradiction.

This concludes the proof of Theorem~\ref{thm:C3C5-free}. Since bipartite graphs are a subclass of~$\{C_3,C_5\}$-free graphs, we obtain the following direct consequence.

\begin{corollary}\label{cor:bipartite}
If~$G$ is an ordered bipartite graph, then a constrained outer-string instance for~$G$ is solvable if and only if it does not contain the alternating-edge obstruction.
\end{corollary}

Testing whether a given ordered graph~$G$ contains the alternating-edge obstruction can be trivially done in polynomial time, yielding the next corollary.

\begin{corollary}\label{cor:polynomial}
For an ordered~$\{C_3,C_5\}$-free graph~$G$, we can decide in polynomial time whether~$G$ has a constrained outer-string representation.
\end{corollary}


\section{NP-Hardness of Recognition of Outer-1-String Graphs}\label{sec:recog-outer-1-string}

Recall that an outer-string representation of a graph~$G$ in which each two curves cross at most~$k$ times is called an outer-$k$-string representation of~$G$. In this section, we prove that it is NP-hard to decide whether a given graph~$G$ admits an outer-1-string representation.
This answers a question of Biedl et al. \cite{biedlCKR:gd24,biedlCKR:DAM26}. (Note that NP-membership is straightforward.) In fact, we prove the following stronger result.

\begin{theorem}\label{thm:nphardness}
Let~$\mathcal{G}$ be a class of graphs that contains all outer-1-string graphs and is a subclass of outer-string graphs. It is then NP-hard to decide whether a given graph~$G$ is in~$\mathcal{G}$.
\end{theorem}

The rest of this section is devoted to the proof of Theorem \ref{thm:nphardness}. 

We construct a reduction from \textsc{Monotone Not-all-equal 3-SAT}, which is known to be NP-complete \cite{Schaefer-STOC78}.
Formally, we are given a Boolean formula~$\varphi$ that consists of~$m$ clauses~$C_1,\dots,C_m$, where each clause is the disjunction of exactly three non-negated literals taken from the set of Boolean variables~$\{v_1,v_2,\dots,v_n\}$.  The goal is to find an assignment of the variables such that every clause has at least one but no more than two variables that are \textsc{True}. 
Assume that we are given an instance~$\varphi$ of \textsc{Monotone Not-all-equal 3-SAT}. We will additionally assume that in~$\varphi$, each variable is contained in at least two clauses. Note that this assumption does not alter the fact that the problem is NP-complete.
From~$\varphi$ we construct a graph~$G$ satisfying the following:

\begin{itemize}
    \item If~$\varphi$ admits a satisfying truth assignment, then~$G$ admits an outer-1-string representation. 
    \item If~$G$ admits an outer-string representation, then~$\varphi$ admits a satisfying assignment. 
\end{itemize}

To construct~$G$, we begin with a cycle on~$10(n+m)$\footnote{The constant 10 is chosen to ensure that certain vertices are represented by curves that are anchored sufficiently far away from each other, but it is not optimal. We chose it to have a rounded multiple of~$(n+m)$.} vertices, which we will denote by~$\mathbb{C}$  and whose vertices we denote by~$c_1,c_2,\dots, c_{10(n+m)}$  in the cyclic order of the cycle~$\mathbb{C}$. To each variable~$v_i$ and each clause~$C_j$, we associate a vertex in~$\mathbb{C}$ as follows: to a variable~$v_i$, we associate the vertex~$c_{10i}$, and to a clause~$C_j$, we associate the vertex~$c_{10(n+j)}$. 

Let us describe the remaining vertices of~$G$. For every variable~$v\in\{v_1,\dotsc,v_n\}$,~$G$ 
contains two vertices~$R_v$ and~$B_v$, which we refer to as the \emph{red assigner} and 
\emph{blue assigner} of the variable~$v$. For every clause~$C\in\{C_1,\dotsc,C_m\}$,~$G$ 
contains three vertices~$U_C$,~$V_C$ and~$W_C$, which we refer to as the \emph{verifiers} of 
the clause~$C$. For any clause~$C$ and variable~$v$ contained in~$C$,~$G$ contains a pair of 
vertices~$r(v,C)$ and~$b(v,C)$, referred to as the \emph{red connector} and \emph{blue 
connector} of the pair~$(v,C)$.

We now describe the edges of~$G$. Recall that the vertices~$c_1,\dotsc,c_{10(m+n)}$ form a 
cycle in~$G$. For a variable~$v_i$, the vertex~$c_{10i}$, which is associated to~$v_i$, is 
adjacent to all the connectors~$r(v_i,C)$ and~$b(v_i,C)$ for all the clauses~$C$ that 
contain~$v_i$. Similarly, for every clause~$C_j$, the vertex~$c_{10(n+j)}$ is adjacent to all 
the connectors~$r(v,C_j)$ and~$b(v,C_j)$ for all the variables~$v$ in~$C_j$. 

Let~$v\in\{v_1,\dotsc,v_n\}$ be a variable, let~$\ell\ge 2$ be the number of clauses containing~$v$, and let~$C_{x_1}(v), C_{x_2}(v), \dots, C_{x_{\ell}}(v)$ be the clauses containing~$v$, with~$1\le x_1<x_2<\cdots <x_\ell\le m$. Let us write~$r_i$ for~$r(v,C_{x_i})$ and~$b_i$ for~$b(v,C_{x_i})$. Then the red connectors~$r_1,\dotsc,r_\ell$ form a clique in~$G$, and so do the blue connectors~$b_1,\dotsc,b_\ell$. Moreover,~$R_v$ is adjacent to the red connectors~$r_1,\dotsc,r_\ell$, while~$B_v$ is adjacent to~$b_1,\dotsc,b_\ell$. In addition,~$r_i$ is adjacent to~$b_j$ if and only if~$i\neq j$. The subgraph of~$G$ induced by the~$2\ell+2$ vertices~$R_v, B_v, r_1,\dotsc,r_\ell, b_1,\dotsc,b_\ell$ is called the \emph{variable gadget} of the variable~$v$, and is denoted by~$H(v)$; see Figure~\ref{fig:variablegadget}.

\begin{figure}
    \centering
    \includegraphics[page=5]{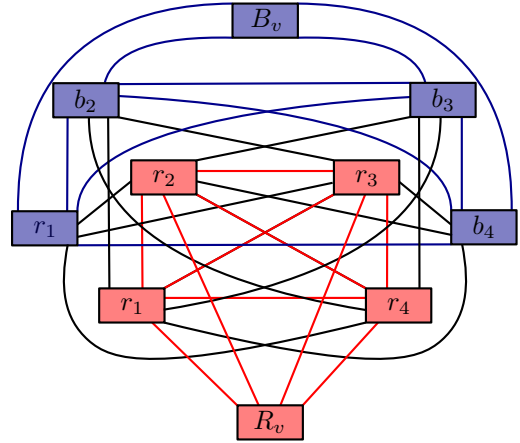}
    \caption{Variable gadget~$H(v)$, with~$\ell=4$.}
    \label{fig:variablegadget}
\end{figure}

Let now~$C\in\{C_1,\dotsc,C_m\}$ be a clause, and let~$v_{a_1}$,~$v_{a_2}$ and~$v_{a_3}$ be the three variables in~$C$, with~$a_1<a_2<a_3$. Let us now write~$r_i$ and~$b_i$ for the connectors~$r(v_{a_i},C)$ and~$b(v_{a_i},C)$, for~$i\in\{1,2,3\}$. Then~$U_C$ is adjacent to~$r_1$ and~$b_2$,~$V_C$ is adjacent to~$r_2$ and~$b_3$, and~$W_C$ is adjacent to~$r_3$ and~$b_1$. The red connectors~$r_1$,~$r_2$ and~$r_3$ form a clique, and so do the blue connectors~$b_1$,~$b_2$ and~$b_3$. Furthermore,~$r_i$ is adjacent to~$b_j$ if and only if~$i\neq j$.  The subgraph of~$G$ induced by the nine vertices~$U_C$,~$V_C$,~$W_C$,~$r_1$,~$r_2$,~$r_3$,~$b_1$,~$b_2$ and~$b_3$ is the \emph{clause gadget} of the clause~$C$, denoted~$K(C)$; see Figure~\ref{fig:clausegadget}.
Note that if two connectors are adjacent in the clause gadget, they belong to different variable gadgets.

To complete the description of~$G$, we need to describe the adjacencies between connectors corresponding to different clauses and different variables. Let~$v_a$ and~$v_b$ be two distinct variables such that~$a<b$, and let~$C_x$ and~$C_y$ be two clauses such that~$C_x$ contains~$v_a$ and~$C_y$ contains~$v_b$. If~$x<y$, then both~$r(v_a,C_x)$ and~$b(v_a,C_x)$ are adjacent to both~$r(v_b,C_y)$ and~$b(v_b,C_y)$. On the other hand, if~$x>y$, there is no adjacency between the connectors of~$(v_a,C_x)$ and those of~$(v_b,C_y)$.

There are no other vertices or edges in~$G$ except for those described above.

\begin{figure}[ht]
    \centering
    \includegraphics[page=6]{Figures/reduction.pdf}
    \caption{Clause gadget~$K(C)$.}
    \label{fig:clausegadget}
\end{figure}

\begin{lemma} \label{lemma:firstdirection}
    If~$\varphi$ has a satisfying assignment, then~$G$ admits an outer-1-string representation. 
\end{lemma}

\begin{proof}
    We will represent the vertices by curves grounded on a horizontal line~$\ell$, with each two curves crossing at most once.  We do this so that we can talk about two vertices being anchored to the left or to the right from each other. Further, we 
     represent vertices of~$\mathbb{C}$ in such a way that the anchors~$c_2^{\bot},\dots, c_{10(n+m)}^{\bot},c_1^{\bot}$ are seen in this left-to-right order starting from~$c_2^\bot$. Such a representation of~$\mathbb{C}$ can be seen in Figure \ref{fig:cyclerepresentation}. 
      
    \begin{figure}[h!]
        \centering
        \includegraphics[page=1, width=0.9\linewidth]{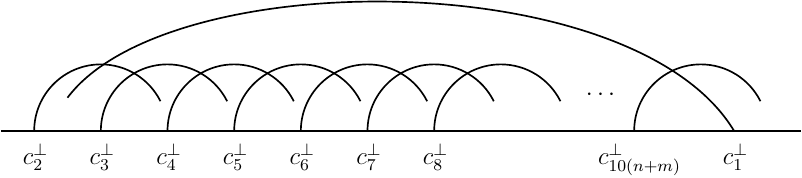}
        \caption{Outer-1-string representation of~$\mathbb{C}$ in the proof of Lemma \ref{lemma:firstdirection}.}
        \label{fig:cyclerepresentation}
    \end{figure}
    
    For each variable~$v_i$ we will represent~$H(v_i)$ in such a way that the anchors of all curves representing the vertices of~$H(v_i)$ are placed between~$c_{10i}^{\bot}$ (the vertex of~$\mathbb{C}$ assigned to~$v_i$) and~$c_{10i+1}^{\bot}$.
    Next, if a variable~$v_i$ is assigned \textsc{True} by the truth assignment, we place~$R_{v_i}^{\bot}$ to the left of ~$B_{v_i}^{\bot}$ and if it was assigned \textsc{False}, we place~$R_{v_i}^{\bot}$ to the right of ~$B_{v_i}^{\bot}$. 
    If~$v_i$ was assigned \textsc{True}, we place the anchors~$r_1^{\bot},b_1^{\bot},r_2^{\bot},b_2^{\bot},\dots, r_l^{\bot},b_l^{\bot}$ in this order between~$R_{v_i}^{\bot}$ and~$B_{v_i}^{\bot}$, if~$v_i$ was assigned \textsc{False} we place the anchors by flipping the position of~$r_i^{\bot}$ and~$b_i^{\bot}$ for every~$i\in \{1,2,,\dots, l\}$. Finally, we represent~$H(v_i)$ as in Figure \ref{fig:variablegadgetrepresentation}.

    \begin{figure}[ht]
        \centering
        \includegraphics[page=2,width=0.9\linewidth]{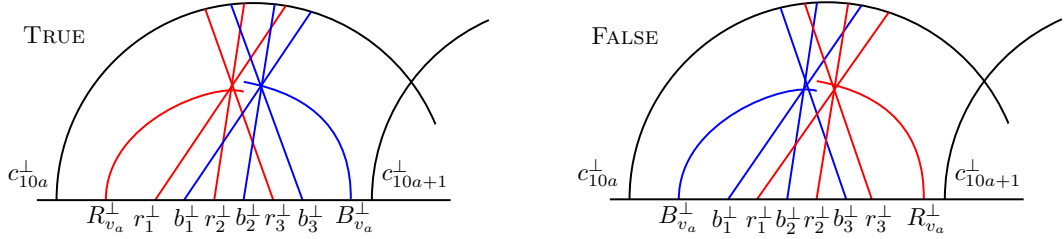}
        \caption{Representation of the variable gadget in the proof of Lemma \ref{lemma:firstdirection}.}
        \label{fig:variablegadgetrepresentation}
    \end{figure}
    
    Note that the representation in Figure \ref{fig:variablegadgetrepresentation} is an outer-1-string representation of~$H(v_i)$, but it  captures neither the adjacencies between vertices of~$H(v_i)$ and~$H(v_j)$ for~$i\neq j$ nor the adjacencies between the vertices of variable and clause gadgets. 

    We now describe how to connect variable gadgets to clause gadgets. Assume that 
    \begin{itemize}
        \item~$1\le a<b<c \le n$,
        \item~$v_a,v_b,v_c \in C_x$ for some~$x \in \{1,2,\dots, m\}$, and
        \item  we represented~$H(v_a), H(v_b), H(v_c)$ as described above.
    \end{itemize}
    Now, for~$i\in \{a,b,c\}$, let~$r_i,b_i$ be the red and blue connectors of~$v_i$ and~$C_x$. We then extend the curve~$r_i^*$ from the current endpoint (which is the crossing with~$c_{10i}^*$) by drawing a piecewise linear curve consisting of two line segments starting at the current endpoint and we stop drawing when we reach a crossing point with~$c_{10(n+x)}^*$ (the curve representing the vertex of~$\mathbb{C}$ assigned to~$C_x$). We extend the representation of~$b_i$ in a similar way. We do this in such a way that the curves~$r_a,b_a,r_b,b_b,r_c,b_c$ are all pairwise disjoint for now; see Figure \ref{fig:vartoclause}. 

    \begin{figure}[ht]
        \centering
        \includegraphics[width=0.9\linewidth,page=3]{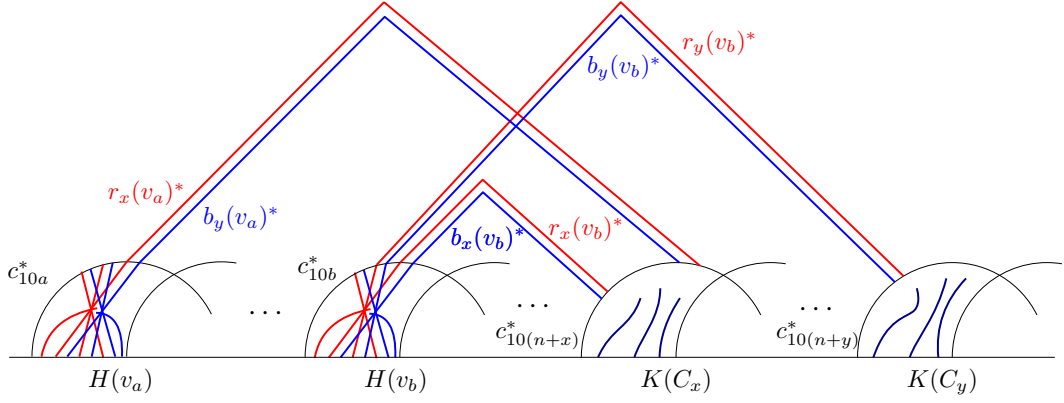}
        \caption{Connecting variable and clause gadgets in the proof of Lemma \ref{lemma:firstdirection}.} 
        \label{fig:vartoclause}
    \end{figure}

    Note that the above construction correctly represents the edges induced by each variable gadget, and it also represents all the edges between the connectors that share neither the same variable nor the same clause. 

    It remains to represent the clause gadgets. To do this, for a clause~$C_x$, we place~$U_{C_x}^{\bot}$,~$V_{C_x}^{\bot}$ and~$W_{C_x}^{\bot}$ between~$c_{10(n+x)}^{\bot}$ and~$c_{10(n+x)+1}^{\bot}$. Now, the order of these anchors and specific representations will depend on the truth assignment of the variables in~$C_x$. See Figure \ref{fig:clauserep} for the representations in each of these cases. Representing all clause gadgets in this way gives us an outer-1-string representation of~$G$. \qedhere

    \begin{figure}[ht]
        \centering
        \includegraphics[page=4,width=0.9\linewidth]{Figures/reduction.pdf}
        \caption{Outer-1-string representation of the clause gadget from the proof of Lemma \ref{lemma:firstdirection} depending on the possible truth  assignments of the variables. }
        \label{fig:clauserep}
    \end{figure}
    \end{proof}

Now, we will argue that if we are given an outer-string representation of~$G$, we can read off a satisfying truth assignment for~$\varphi$. To do this, we make a few observations about the properties of any such representation. 

First, we observe that in any outer-string observation of~$\mathbb{C}$, we encounter the anchors of the even vertices in the expected order. We state this formally below. 

\begin{lemma} \label{lemma:cyclerepresentation}
    Let~$k\ge 6$ be even, and let~$C_k$ be the cycle with vertices denoted~$c_1,\dotsc,c_k$ in the order of appearance on the cycle. Then in any outer-string representation of~$C_k$, the anchors of the even-numbered vertices~$c_2^\bot, c_4^\bot,\dotsc, c_k^\bot$ appear along the boundary of~$D$ either clockwise or counter-clockwise in the order given above.
\end{lemma}

\begin{proof}
 Fix an outer-string representation of~$C_k$. We may assume that~$c_2^\bot$,~$c_4^\bot$ and~$c_6^\bot$ form a counter-clockwise triple along the boundary of~$D$, otherwise we replace the representation with its mirror image. Let~$j$ be the largest even index such that~$c_2^\bot, c_4^\bot,\dotsc,c_j^\bot$ form a counter-clockwise sequence. If~$j=k$ we are done, so suppose for contradiction that~$j<k$. For every even~$i\in\{2,4,6,\dotsc,j-2\}$, let~$A_i$ denote the counter-clockwise arc of the boundary of~$D$ starting in~$c_i^\bot$ and ending in~$c_{i+2}^\bot$. By the choice of~$j$, we know that~$c_{j+2}^\bot$ is inside~$A_i$ for some~$i\in\{2,4,6,\dotsc,j-2\}$.

If~$c_{j+2}^\bot$ is in~$A_2$, then the paths~$P=c_2 c_3 c_4$ and~$Q=c_j c_{j+1} c_{j+2}$ are in contradiction with Lemma~\ref{lem:crossingpaths}. On the other hand, if~$c_{j+2}^\bot$ is in 
$A_4\cup A_6\cup\dotsb\cup A_{j-2}$, we consider the paths~$P=c_{j+2} c_{j+3} \dotsb c_{k-1} c_k c_1 c_2$ and~$Q=c_4 c_5 c_6 \dotsb c_j$, again contradicting Lemma~\ref{lem:crossingpaths}.
\end{proof}
Now, let us fix an outer-string representation~$G^*$ of~$G$. In view of Lemma \ref{lemma:cyclerepresentation}, we can assume that the anchors~$c_2^\bot,c_4^{\bot},\dots, c_{10(n+m)}^{\bot}$ are seen in this order when traversing the boundary of~$D$ counter-clockwise; if they were ordered clockwise, we would replace the representation by its mirror image. 
For an even~$i\in\{2,4,\dotsc, 10(m+n)\}$, let~$A_i$ denote the counter-clockwise arc of the boundary of~$D$ starting in~$c_i^\bot$ and ending in~$c_{i+2}^\bot$, with~$A_{10(m+n)}$ starting in~$c_{10(m+n)}^\bot$ and ending in~$c_2^\bot$. 

Recall that in the construction of~$G$, the variable~$v_i$ has been associated to the vertex~$c_{10i}$ and the clause~$C_j$ to the vertex~$c_{10(n+j)}$. We now define the \emph{region of~$v_i$} as~$A_{10i-2}\cup A_{10i}$, and the \emph{region of~$C_j$} as~$A_{10(n+j)-2}\cup A_{10(n+j)}$. Notice that the regions of all the variables and clauses are pairwise disjoint.

\begin{lemma} \label{lem:anchorsofmainvertices}
For each variable~$v$, the anchors of its two assigners~$R_v$ and~$B_v$ are in the region of~$v$, while for every clause~$C$, the anchors of its three verifiers~$U_C$,~$V_C$ and~$W_C$ are in the region of~$C$.
\end{lemma}
\begin{proof}
For contradiction, assume first that an assigner of~$v$, say~$R_v$, is anchored inside an arc~$A_k$ which is not part of the region of~$v$. Choose a clause~$C_j$ containing the variable~$v$ and such that~$A_k$ is not part of the region of~$C_j$. This choice is possible, because by our assumption on~$\varphi$, there are at least two clauses containing~$v$. Now consider the path~$P=c_k, c_{k+1}, c_{k+2}$ and the path~$Q=R_v, r(v,C_j), c_{10(n+j)}$ and observe that these paths contradict Lemma~\ref{lem:crossingpaths}.

A similar argument shows that each of the three verifiers~$U_C$,~$V_C$ and~$W_C$ are anchored in the region of~$C$.
\end{proof}

Finally, we describe how to read off the truth assignment from~$G^*$.  To a variable~$v_a \in \varphi$, we assign \textsc{True}, if the anchor of~$R_{v_a}$ appears before the anchor of~$B_{v_a}$ when we pass through the region of~$v_a$ in counter-clockwise direction, otherwise~$v_a$ is assigned \textsc{False}.  The next lemma tells us that this defines a satisfying assignment for~$\varphi$. 

\begin{lemma}
Let~$C$ be a clause in~$\varphi$, and let~$v_a,v_b,v_c$ be the variables in~$C$, with~$a<b<c$. Then the truth assignment described above does not assign the same value to all of~$v_a,v_b,v_c$.     
\end{lemma}

\begin{proof}
    Assume, for contradiction, that~$v_a,v_b,v_c$ are all assigned \textsc{True} (if they are assigned \textsc{False}, a similar argument applies). This means that for each~$i\in \{a,b,c\}$,~$R_{v_i}^\bot$ precedes~$B_{v_i}^\bot$ in the counter-clockwise order in the region of~$v_i$. 
    We now argue that it is impossible to represent~$U_{C}$,~$V_{C}$ and~$W_{C}$ in~$G^*$. 
Consider the two paths~$P=R_{v_a}, r(v_a,C), U_C$ and~$Q=B_{v_a},b(v_a,C),W_C$. The two paths satisfy the assumptions of Lemma~\ref{lem:crossingpaths}. The lemma then implies that the anchor of~$W_C$ must precede the anchor of~$U_C$ when the region of~$C$ is passed counter-clockwise. Applying the same argument to~$v_b$, we obtain that~$U_C$ must precede~$V_C$, and applying the argument to~$v_c$, we conclude that~$V_C$ must precede~$W_C$, which is impossible.
\end{proof}

This concludes the proof of Theorem \ref{thm:nphardness}.

\begin{corollary}
For every fixed~$k\ge 1$, recognizing outer-$k$-string graphs is NP-complete.     
\end{corollary}

\section{Conclusion and Open Problems}\label{sec:conclusion}

In this paper we considered outer-string representations of graphs with additional restrictions. We characterized  ordered~$\{C_3,C_5\}$-free graphs which admit a constrained outer-string representation. Additionally, we answered a question posed in \cite{biedlCKR:gd24,biedlCKR:DAM26} by proving that the recognition of outer-1-string graphs is NP-hard. There are several directions in which our results can be extended; we mention some of them here. 

Firstly, an easy consequence of Theorem \ref{thm:C3C5-free} is that a bipartite  graph (or, more generally, a~$\{C_3,C_5\}$-free graph)~$G$ admits an outer-string representation if and only if it is possible to order the vertices of~$G$ so that the resulting ordering introduces no alternating-edge obstructions. This means that if we are given an ordering on~$G$, we can decide in polynomial time whether~$G$ admits a constrained outer-string representation. However, we do not know if we can decide if such an ordering exists.

\begin{problem}\label{pro-bipartite}
Let~$G$ be a bipartite graph. Can we decide in polynomial time whether~$G$ has a cyclic ordering of vertices that has no alternating-edge obstruction?
\end{problem}
In view of Theorem~\ref{thm:C3C5-free}, a positive answer to Problem~\ref{pro-bipartite} is equivalent to a polynomial algorithm for testing whether a bipartite graph is an outer-string graph. 

For a graph~$G$ that is not necessarily bipartite, recent results by Kun and Ne\v{s}et\v{r}il \cite{kun2026} imply that it is NP-hard to decide whether~$G$ admits an ordering with no alternating-edge obstruction. However, their reduction does not work if we assume that the graph is bipartite.  

As we mentioned before, co-holes are outer-string obstructions; moreover, Sinden~\cite{sinden66} has given an example of a co-hole-free ordered graph with no constrained outer-string representation, see Figure~\ref{fig-sinden}.

\begin{figure}[ht]
    \centering
    \includegraphics[scale = 0.9]{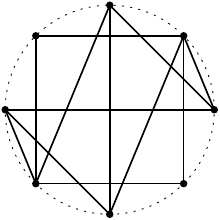}
    \caption{Sinden's example of a co-hole-free ordered graph with no constrained outer-string representation}
    \label{fig-sinden}
\end{figure}


It would be interesting to find out if the construction in Figure \ref{fig-sinden} can be generalized to higher number of vertices and to determine if any other outer-string obstructions exist.

Another intriguing problem is to find other classes of graphs for which the co-holes are the only obstructions. Due to Proposition \ref{prop:testing-co-holes} below, this would imply a polynomial algorithm for the recognition of constrained outer-string graphs in these classes.

\begin{proposition}\label{prop:testing-co-holes}
It can be tested in polynomial time whether  an ordered graph contains a co-hole of length at least 4.   
\end{proposition}

\begin{proof}
    We describe an algorithm for testing if a given ordered graph contains a co-hole of length at least 4. Let~$V(G)=\{v_1,v_2,\ldots,v_n\}$ and let this be the  cyclic ordering of~$G$ given as part of the input. We will first test whether~$G$ contains a co-hole which contains the vertex~$v_1$. Such a co-hole would involve vertices~$v_x,v_y$ such that~$1<x<y$,~$v_1v_x,v_1v_y\notin E(G)$ and~$v_xv_y\in E(G)$. For every pair of such vertices, consider the complement~$-G$ of~$G$, orient its edges~$v_iv_j$ from~$v_i$ to~$v_j$ when~$i<j$ and ask if it contains a directed path from~$v_x$ to~$v_y$ which only uses vertices adjacent (in~$G$) to~$v_1$. This can be answered by standard search algorithms in linear time. If there is no such path,~$G$ does not contain a co-hole involving the vertices~$v_y,v_1,v_x$ (in this consecutive order). If there is such a path, consider a shortest one, say~$P$. The shortest path is induced, and hence together with~$v_1$, it forms a co-hole in~$G$. Since~$v_xv_y$ is an edge of~$G$, the path~$P$ contains another vertex~$v_z$ with~$x<z<y$ and he length of the co-hole is at least 4.

For fixed~$v_1,v_x,v_y$, the running time is upper bounded by~$O(n^2)$. Testing this for all vertices in place of~$v_1$ and respective~$v_x,v_y$, we get running time~$O(n^5)$.
\end{proof}

\begin{problem}
    Find other classes of graphs for which the co-holes are the only outer-string obstructions.
\end{problem}

Unfortunately, it seems that extending Theorem \ref{thm:C3C5-free} to a broader class of graphs would require a significantly different approach. In Figure \ref{fig:unextendable}, we show an example of an ordered graph~$G$ that admits a constrained outer-string representation. We also show that for 
the vertex~$t\in V(G)$, there is an outer-string representation of~$G - t$ that cannot be extended to a representation of~$G$ using the approach from the proof of Theorem~\ref{thm:C3C5-free}. 

\begin{figure}[h]
    \centering
    \includegraphics[scale = 0.75]{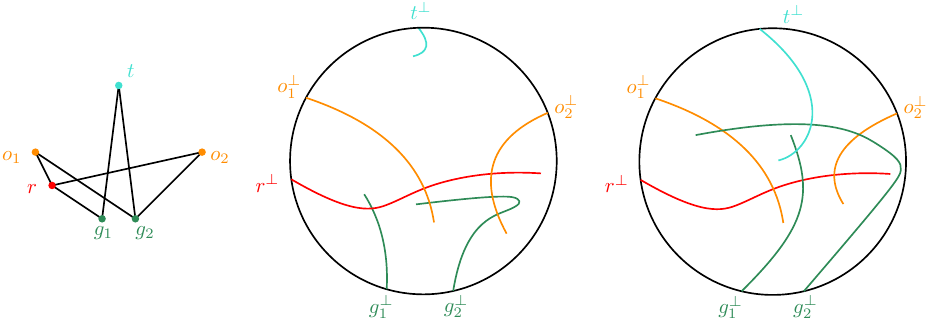}
    \caption{Left: an ordered graph~$G$ ; center: a representation of~$G-t$ which cannot be extended to a representation of~$G$ by the method of trajectories and tunnels used to prove Theorem~\ref{thm:C3C5-free}; right: a constrained outer-string representation of~$G$.}
    \label{fig:unextendable}
\end{figure}

\bibliography{references.bib}

@inproceedings{biedlCKR:gd24,
  author       = {Therese Biedl and
                  Sabine Cornelsen and
                  Jan Kratochv{\'{\i}}l and
                  Ignaz Rutter},
  editor       = {Stefan Felsner and
                  Karsten Klein},
  title        = {Constrained Outer-String Representations},
  booktitle    = {32nd International Symposium on Graph Drawing and Network Visualization (GD'24)},
  series       = {LIPIcs},
  volume       = {320},
  pages        = {10:1--10:18},
  publisher    = {Schloss Dagstuhl - Leibniz-Zentrum f{\"{u}}r Informatik},
  year         = {2024},
  OPTurl          = {https://doi.org/10.4230/LIPIcs.GD.2024.10},
  doi          = {10.4230/LIPICS.GD.2024.10}
}

@article{biedlCKR:DAM26,
  author       = {Therese Biedl and
                  Sabine Cornelsen and
                  Jan Kratochv{\'{\i}}l and
                  Ignaz Rutter},
    title        = {Constrained Outer-String Representations},
  journal    = {Discrete Applied Mathematics},
    volume       = {390},
  pages        = {114--134},
    year         = {2026},
    doi          = {10.1016/j.dam.2026.04.018}
}

@article{Kratochvil91a,
  author       = {Jan Kratochv{\'{\i}}l},
  title        = {String graphs. {II.} recognizing string graphs is {NP}-hard},
  journal      = {J. Comb. Theory {B}},
  volume       = {52},
  number       = {1},
  pages        = {67--78},
  year         = {1991},
  url          = {https://doi.org/10.1016/0095-8956(91)90091-W},
  doi          = {10.1016/0095-8956(91)90091-W},
  bibsource    = {dblp computer science bibliography, https://dblp.org}
}

@book{B-L-Spinrad,
  author       = {Brandst{\"a}dt, Andreas and Le, Van Bang and Spinrad, Jeremy P.},
  title        = {{Graph Classes: A Survey}},
    publisher    = {Society for Industrial
and Applied Mathematics},
  year         = {1999},
doi ={10.1137/1.9780898719796}
}

@article{MiddendorfP93cylinder,
  author       = {Matthias Middendorf and
                  Frank Pfeiffer},
  title        = {Weakly transitive orientations, {H}asse diagrams and string graphs},
  journal      = {Discret. Math.},
  volume       = {111},
  number       = {1-3},
  pages        = {393--400},
  year         = {1993},
  url          = {https://doi.org/10.1016/0012-365X(93)90176-T},
  doi          = {10.1016/0012-365X(93)90176-T},
  bibsource    = {dblp computer science bibliography, https://dblp.org}
}

@article{KratochvilM91exp,
  author       = {Jan Kratochv{\'{\i}}l and Ji{\v{r}{\'{\i}} Matou{\v s}ek}},
  title= {String graphs requiring exponential representations},
  journal      = {J. Comb. Theory {B}},
  volume       = {53},
  number       = {1},
  pages        = {1--4},
  year         = {1991},
  url          = {https://doi.org/10.1016/0095-8956(91)90050-T},
  doi          = {10.1016/0095-8956(91)90050-T},
  bibsource    = {dblp computer science bibliography, https://dblp.org}
}

@book{McKee,
    author = { McKee, Terry A. and McMorris,  F. R. },
    title = {{Topics in Intersection Graph Theory}},
    publisher = {Society for Industrial
and Applied Mathematics},
    year = {1999}
}

@article{ehrlichET76,
title = {Intersection graphs of curves in the plane},
journal = {Journal of Combinatorial Theory, Series B},
volume = {21},
number = {1},
pages = {8-20},
year = {1976},
doi = {https://doi.org/10.1016/0095-8956(76)90022-8},
OPTurl = {https://www.sciencedirect.com/science/article/pii/0095895676900228},
author = {Gideon Ehrlich and Shimon Even and Robert Endre Tarjan},
}

@book{golumbic:perfect,
    author = {Martin Charles Golumbic},
    title = {Algorithmic Graph Theory and Perfect Graphs},
    publisher = {Elsevier},
    year = 1980
}

@ARTICLE{sinden66,
    AUTHOR = "Frank W. Sinden",
    TITLE = "Topology of thin film {RC} circuits",
    JOURNAL = "Bell System Tech. J.",
    VOLUME = {45},
    PAGES = {1639-1662},
    YEAR = {1966},
    doi = {10.1002/j.1538-7305.1966.tb01713.x}
    }

@inproceedings{BiedlBD18,
  author       = {Therese Biedl and
                  Ahmad Biniaz and
                  Martin Derka},
  title        = {On the Size of Outer-String Representations},
  booktitle    = {16th Scandinavian Symposium and Workshops on Algorithm Theory (SWAT'18)},
  editor = {David Eppstein},
  series       = {LIPIcs},
  volume       = {101},
  pages        = {10:1--10:14},
  publisher    = {Schloss Dagstuhl - Leibniz-Zentrum f{\"{u}}r Informatik},
  year         = {2018},
  doi          = {10.4230/LIPICS.SWAT.2018.10},
}

@inproceedings{BiedlD17,
  author       = {Therese Biedl and
                  Martin Derka},
  title        = {Order-Preserving 1-String Representations of Planar Graphs},
  booktitle    = {International Conference on Current Trends in Theory and Practice of Computer Science (SOFSEM'17)},
  editor       = {Bernhard Steffen and
                  Christel Baier and
                  Mark van den Brand and
                  Johann Eder and
                  Mike Hinchey and
                  Tiziana Margaria},
  series       = {Lecture Notes in Computer Science},
  volume       = {10139},
  pages        = {283--294},
  publisher    = {Springer},
  year         = {2017},
  doi          = {10.1007/978-3-319-51963-0_22},
}

@inproceedings{kratochvil:82,
 author={Jan Kratochv{\'\i}l},
 title={String graphs},
 booktitle={Graphs and Other Combinatorial Topics, Proceedings Third Czechoslovak Symposium on Graph Theory, Prague},
 year = 1982, 
 publisher = {Teubner, Berlin},
 series = {Teubner Texte zur Mathematik},
 volume = {59},
 pages={168--172}
}

@article{jcss-SchaeferSS03,
  author       = {Marcus Schaefer and
                  Eric Sedgwick and
                  Daniel {\v S}tefankovi{\v c}},
  title        = {Recognizing string graphs in {NP}},
  journal      = {J. Comput. Syst. Sci.},
  volume       = {67},
  number       = {2},
  pages        = {365--380},
  year         = {2003},
  url          = {https://doi.org/10.1016/S0022-0000(03)00045-X},
  doi          = {10.1016/S0022-0000(03)00045-X},
  bibsource    = {dblp computer science bibliography, https://dblp.org}
}

@inproceedings{Schaefer-STOC78,
  author       = {Thomas J. Schaefer},
  editor       = {Richard J. Lipton and
                  Walter A. Burkhard and
                  Walter J. Savitch and
                  Emily P. Friedman and
                  Alfred V. Aho},
  title        = {The Complexity of Satisfiability Problems},
  booktitle    = {Proceedings of the 10th Annual {ACM} Symposium on Theory of Computing,
                  May 1-3, 1978, San Diego, California, {USA}},
  pages        = {216--226},
  publisher    = {{ACM}},
  year         = {1978},
    doi          = {10.1145/800133.804350}
  }

@inproceedings{kun2026,
author = {G{\'{a}}bor Kun and Jaroslav Ne{\v{s}}et{\v{r}}il},
title = {Dichotomy for orderings?},
booktitle = {Proceedings of the 2026 Annual ACM-SIAM Symposium on Discrete Algorithms (SODA)},
pages = {4175--4187},
year = {2026},
doi = {10.1137/1.9781611978971.153},
publisher    = {{ACM}},
URL_ = {https://epubs.siam.org/doi/abs/10.1137/1.9781611978971.153},
eprint_ = {https://epubs.siam.org/doi/pdf/10.1137/1.9781611978971.153}
}

\newpage
\appendix


\end{document}